\documentclass[12pt]{article}
\usepackage{amssymb}
\usepackage[latin1]{inputenc}
\usepackage[T1]{fontenc}
\usepackage{amsthm}
\usepackage{amsmath}
\usepackage{eso-pic}
\usepackage{graphicx}

\usepackage{color}
\newtheorem{thm}{Theorem}
\theoremstyle{definition}
\newtheorem{defn}{Definition}

\newtheorem{cor}{Corollary}
\newtheorem{prop}{Proposition}

\begin{document}

\title{Progress Towards  the Conjecture on APN Functions and Absolutely Irreducible Polynomials }
\author{Moises~Delgado, Heeralal~Janwa
\thanks{Department of Mathematics, University of Puerto Rico, Rio
Piedras Campus, San Juan PR, USA., \textcolor{black}{moises.delgado@upr.edu; heeralal.janwa@upr.edu}
}}
\date{{}}
\maketitle

\begin{abstract}
Almost Perfect Nonlinear (APN) functions are very useful in cryptography, when \textcolor{black}{they}  are used as
S-Boxes, because of their good resistance to differential
cryptanalysis. An APN function $f:\mathbb{F}_{2^n}\rightarrow
\mathbb{F}_{2^n}$ is called exceptional APN if it is APN on
infinitely many extensions of $\mathbb{F}_{2^n}$. Aubry, McGuire and Rodier
conjectured that the only exceptional APN functions are the Gold and
the Kasami-Welch monomial functions. They established  that a
polynomial function of odd degree is not exceptional APN provided
the degree is not a Gold number $(2^k+1)$ or a Kasami-Welch number
$(2^{2k}-2^k+1)$. When the degree of the polynomial function is a Gold
number, several partial results have been obtained \textcolor{black}{\cite{Aub,CauF,DJ, {DJDESI},Rod2}.} 
\textcolor{black}{One of the results in this article is a  proof of  the relatively primeness of the multivariate APN polynomial conjecture, in the Gold degree case.}
This helps us extend substantially previous  results. We prove that Gold degree
polynomials of the form $x^{2^k+1}+h(x)$, where
$\deg(h)$ is any odd integer (with the natural exceptions),  can not be exceptional APN.

We also show absolute irreducibility of several  classes of multivariate polynomials  over finite fields and discuss their applications.

\end{abstract}



\section{\textbf{INTRODUCTION}}
Block ciphers are symmetric key algorithms for performing encryption or
decryption. Block ciphers map a block of bits to another block of
bits in such a way that it is difficult to guess the mapping. For block ciphers, the
effectiveness of the main cryptanalysis techniques can be measured by
some quantities related to the round of encryption, usually named substitution box (S-box).\\
\textcolor{black}{For differential attacks, the attacker is able to select inputs and examine outputs in an attempt
to derive the secret key, more exactly, the attacker will select pairs of inputs $x$, $y$ satisfying
a particular $a=x-y$, knowing that for that $a$ value, a particular $b=f(x)-f(y)$ value
occurs with high probability.\\
Then, one of the desired properties for an S-box to have high
resistance against differential attacks is that, given any
plaintext difference $x-y=a$, it provides a ciphertext difference
$f(x)-f(y)=b$ with small probability.}\\

\begin{defn} Let $L=\mathbb{F}_q$, with $q=2^n$ for some positive
integer $n$. A function $f:L\rightarrow L$ is said to be {\it almost
perfect nonlinear }(APN) on $L$ if for all $a,b \in L$, $a \neq 0$,
the equation
\begin{equation}
f(x+a)-f(x)=b
\end{equation}
have at most 2 solutions.
\end{defn}
\vspace{0.3cm} Equivalently, $f$ is APN if the set $\{f(x+a)-f(x):x
\in L \}$ has size at least $2^{n-1}$ for each $a\in L^{\ast}$.
Moreover, since $L$ has characteristic 2, if $x$ is a solution of the
equation (1), $x+a$ is also a solution. Then the number of solutions
\textcolor{black}{of (1)} must be an even number.\\
The best known examples of APN functions are the Gold functions
$f(x)=x^{2^k+1}$ and the Kasami-Welch functions $f(x)=x^{4^k-2^k+1}$, whose 
names are due to its exponents, the Gold and Kasami-Welch numbers respectively.
These functions are APN on any field $\mathbb{F}_{2^n}$ where $k,n$ are
\textcolor{black}{relatively prime integers}. The function $f(x)=x^{2^r+3}$
(Welch function) is also APN on $\mathbb{F}_{2^n}$, where $n=2r+1$.\\
The APN property is invariant under some transformations of
functions. A function $f:L \rightarrow L$ is linear if
$$\sum_{i=0}^{n-1}c_ix^{2^i}, \,\,\,\,\,\, c_i\in L.$$
The sum of a linear function and a constant is called an affine
function.\\
\textcolor{black}{Two functions $f$ and $g$  are called {\it extended affine equivalent}
\textcolor{black}{(EA-equivalent)}, 
 if
$f=A_1\circ g \circ A_2+A$\textcolor{black}{,} where $A_1$ \textcolor{black}{and} $A_2$ are \textcolor{black}{linear maps
and $A$ is a constant function}. They are called {CCZ-equivalence}, 
if the graph of $f$ can be obtained
from the graph of $g$ by an affine permutation.
EA equivalence is a particular case of CCZ equivalence
and two CCZ equivalent functions preserves the APN property (for
more details see \cite{Car}). In
general, proving CCZ equivalence is very difficult.}\\
\textcolor{black}{APN functions and their applications have become very
important for the mathematicians in the last years. APN functions defined over
$\mathbb{F}_{2^n}$ are related to others mathematical objects, for example
they are equivalent to binary error correcting codes $[2^n,2^n-2n-1,6]$, they are also
equivalent to a certain class of dual hyperovals in the projective geometry.}\\

Until 2006, the list of known affine inequivalent APN functions over
$K=\mathbb{F}_{2^n}$ was the families of monomial functions
$f(x)=x^d$, where the exponent $d$ is as in the following table:

\begin{table}[h]
\centering
\begin{tabular}{|c|c|c|}
  \hline
  $x^d$ & Exponent $d$ & Constraints \\
  \hline
  Gold & $2^r+1$ & $(r,n)=1$ \\
  \hline
  Kasami-Welch & $2^{2r}-2^r+1$& $(r,n)=1$,$n$ odd \\
  \hline
  Welch & $2^r+3$ & $n=2r+1$\\
  \hline
  Niho & $2^r+2^{r/2}-1$ & $n=2r+1$, $r$ even \\
       & $2^r+2^{(3r+1)/2}-1$ & $n=2r+1$, $r$ odd \\
  \hline
  Inverse & $2^{2r}-1$ & $n=2r+1$ \\
  \hline
  Dobbertin & $2^{4r}+2^{3r}+2^{2r}+2^r-1$ & $n=5r$\\
  \hline
\end{tabular}
\end{table}

\textcolor{black}{Mathematicians} conjectured that this list was complete, up to equivalence.
Motivated by this conjecture, several authors worked to find new APN
functions not equivalent to the \textcolor{black}{known} ones. In February 2006,
Y.Edel, G.Kyureghyan and A.Pott \cite{Ede} published a paper with the first
example of an APN function, that is a binomial of degree 36, \textcolor{black}{which is} not
equivalent to any of the functions appeared in the above list. The
function
$$x^3+ux^{36} \in GF(2^{10})[x]$$
where $u\in wGF(2^5)^* \cup w^2GF(2^5)^*$ and $w$ has order 3 in
 $GF(2^{10})$, is APN on $GF(2^{10})$.\\
From the emergence of this first example, there are now several more
families of APN functions inequivalent to \textcolor{black}{monomial} functions. As some
examples, Budaghyan, Carlet and Lender \cite{Bud} found the following
family of quadratic APN functions:
$$f(x)=x^{2^s+1}+wx^{2^{ik}+2^{mk+s}}$$
where $n=3k$, $(k,3)=(s,3k)=1$, $k\geq 4$, $i=sk(mod 3)$, $m=3-i$
and $w$ has the order $2^{2k}+2^k+1$.\\
Notice that, as shown in the above examples, the APN property depends on the extension degree
of $\mathbb{F}_2$. For any $t=2^r+1$ or $t=2^{2r}-2^r+1$ there exist
infinitely many values $m$ such that $(r,m)=1$. That is, any fixed
Gold or Kasami-Welch function \textcolor{black}{which} is APN on $L$ is also APN on
infinitely many extensions of $L$. Such functions are called
\textbf{exceptional} APN functions. One way to face a classification
problem of APN functions is to determine \textcolor{black}{which} APN functions are APN
infinitely often. This problem has been studied for monomials
functions by Janwa, Mc.Guire, Wilson, Jedlika, Hernando \cite{Her,Jan,Jan2} and
more recently for polynomials by \textcolor{black}{Aubry, McGuire, Rodier, Caullery, Delgado and Janwa \cite{Aub,CauF,DJ,Rod2}}.\\

\begin{defn} Let $L=\mathbb{F}_q$, $q=2^n$ for some positive
integer $n$. A function $f:L\rightarrow L$ is called
\textbf{exceptional APN} if $f$ is APN on $L$ and also on infinitely
many extensions of $L$.
\end{defn}
\vspace{0.3cm}
Aubry, McGuire and Rodier \textcolor{black}{conjectured the following in} \cite{Aub}.\\
\textbf{CONJECTURE:} Up to equivalence, the Gold and Kasami-Welch
functions are the only exceptional APN functions.\\
It has been established \cite{Aub} that a polynomial function of
odd degree is not exceptional APN \textcolor{black}{when} the function is
not a Gold function or a Kasami-Welch function. Although
there are \textcolor{black}{some} results for \textcolor{black}{the cases of non-monomial functions
which are polynomials of} Gold and Kasami-
Welch degree, these cases \textcolor{black}{remain} open. In this paper we obtain new results which 
prove that a big infinite family of Gold degree
polynomials can not be exceptional APN.

\textcolor{black}{We make substantial progress towards the resolution of this conjecture. 
One of of our main results in this article is a  proof of  the relatively primeness of the multivariate APN polynomial, in the Gold degree case (see Theorem 10).
This helps us extend substantially previous  results. In particular, we prove that Gold degree
polynomials of the form $x^{2^k+1}+h(x)$, where
$\deg(h)$ is any odd integer (with the natural exceptions),  can not be exceptional APN (section 5). We also show absolute irreducibility of several  classes of multivariate polynomials  over finite fields .
 and discuss their applications (section 5). We also give a proof of the "even case" of another theorem (section 6)}

\section{\textbf{EXCEPTIONAL APN FUNCTIONS AND THE SURFACE $\phi(x,y,z)$}}
Let $L=\mathbb{F}_q$, $q=2^n$ for some positive integer $n$. Rodier \textcolor{black}{characterized APN functions
as follows} \cite{Rod}.\\

\begin{prop} A function $f:L\rightarrow L$ is APN if and
only if \textcolor{black}{the rational points $f_q$ of the affine surface}
$$f(x)+f(y)+f(z)+f(x+y+z)=0$$
are contained in the surface
$(x+y)(x+z)(y+z)=0.$
\end{prop}
\vspace{0.3cm} Given a polynomial function $f \in L[x,y,z]$,
$\deg(f)=d$. We define:
\begin{equation}
\phi(x,y,z)=\frac{f(x)+f(y)+f(z)+f(x+y+z)}{(x+y)(x+z)(y+z)}
\end{equation}
Then $\phi$ is a polynomial over $L[x,y,z]$ of degree $d-3$. This
polynomial defines a surface $X$ in the three dimensional
affine space $L^3$.\\
It can be shown that if $f(x)=\sum_{j=0}^d a_jx_j$, then:
$$\tiny{\phi(x,y,z)=\sum_{j=3}^d a_j\phi_j(x,y,z)}$$
where
\begin{equation}
 \phi_j(x,y,z)=\frac{x^j+y^j+z^j+(x+y+z)^j}{(x+y)(x+z)(y+z)}
\end{equation}
is homogeneous of degree $j-3$.\\
From the above proposition, one can deduce  the next corollary
whose proof can be found in \cite{Rod}.\\

\begin{cor} If the polynomial function $f:L\rightarrow L$ (of degree
$d\geq 5$) is APN and the affine surface $X$
\begin{equation}
\phi(x,y,z)=\frac{f(x)+f(y)+f(z)+f(x+y+z)}{(x+y)(x+z)(y+z)}=0
\end{equation}
is absolutely irreducible, then the \textcolor{black}{projective closure of $X$,
$\overline{X}$ admits} at most $4((d-3)q+1)$ rational points.
\end{cor}
\vspace{0.3cm} Using this corollary and the bound results of
Lang-Weil and Ghorpade-Lachaud, that \textcolor{black}{guarantees} many rational points
on a surface
for all $n$ sufficiently large, we have the following theorem \cite{Rod}.\\

\begin{thm} Let $f:L\rightarrow L$ a
polynomial function of degree $d$. Suppose that the surface $X$ of
affine equation
$$\frac{f(x)+f(y)+f(z)+f(x+y+z)}{(x+y)(x+z)(y+z)}=0$$
is absolutely irreducible (or has an absolutely irreducible
component over $L$) and $d \geq 9$, $d < 0.45q^{1/4}+0.5$, then $f$
is not an APN function.
\end{thm}
\vspace{0.3cm}
 Using this theorem, it can be proved that, if $X$ is
absolutely irreducible (or has an absolutely irreducible factor over
$L$) then $f$ is not exceptional APN.

\section{\textbf{RECENT RESULTS}}
\textcolor{black}{In this section we state some families of polynomial functions that
aren't exceptional APN from \cite{Aub,CauF,DJ,Rod2}}

\begin{thm}\textcolor{black}{(Aubry, McGuire, Rodier \cite{Aub})} If the degree of the polynomial function $f$ is odd and not a
Gold or a Kasami-Welch number then $f$ is not APN over
$L=\mathbb{F}_{q^n}$ for all $n$ sufficiently large.
\end{thm}

For the even degree case, they proved the following:

\begin{thm}
If the degree of the polynomial function $f$ is $2e$ with $e$ odd,
and if $f$ contains a term of odd degree, then $f$ is not APN over
$L=\mathbb{F}_{q^n}$ for all $n$ sufficiently large.
\end{thm}

\begin{thm}\textcolor{black}{(Rodier \cite{Rod2})}
If the degree of the polynomial function $f$ is even such that $\deg(f)=4e$ with
$e\equiv 3(\mod 4)$ and if the polynomials of the form $(x+y)(y+z)(z+x)+P$ with\\
$$P(x,y,z)=c_1(x^2+y^2+z^2)+c_4(xy+xz+yz)+b_1(x+y+z)+d$$
for $c_1,c_4,b_1,d \in \mathbb{F}_{q^3}$, do not divide $\phi$ then $f$ is not APN
over $\mathbb{F}_{q^n}$ for $n$ large.
\end{thm}

\textcolor{black}{Rodier proved} a more precise result for polynomials of degree 12. If the degree of the polynomial
defined over $\mathbb{F}_{q}$ is 12, then either $f$ is not APN over $\mathbb{F}_{q^n}$ for
large $n$ or $f$ is CCZ equivalent to the Gold function $f(x)=x^3$.\\
\textcolor{black}{Recently}, Florian Caullery \cite{CauF} obtained an analogous result for polynomials of degree 20.
They are not exceptional APN or are CCZ equivalent to $f(x)=x^5$.

Aubry, McGuire and Rodier \cite{Aub} also found results for Gold degree polynomials.

\begin{thm} Suppose $f(x)=x^{2^k+1}+g(x)\in L[x]$ where $\deg(g)\leq
2^{k-1}+1$. Let $g(x)=\sum_{j=0}^{2^{k-1}+1}a_jx^j$. Suppose that
there exists a nonzero coefficient $a_j$ of $g$ such that
$\phi_j(x,y,z)$ is absolutely irreducible. Then $\phi(x,y,z)$ is
absolutely irreducible and so $f$ is not exceptional APN.
\end{thm}

Some functions \textcolor{black}{covered} by this theorem are:\\
$f(x)=x^{17}+h(x)$, where $\deg(h)\leq
9$; or $f(x)=x^{33}+h(x)$, where $\deg(h)\leq 17$.\\
\textcolor{black}{Additionally,} they also found that the bound for $g$ is best
possible in the sense that if $f(x)=x^{2^k+1}+g(x)$ with
$\deg(g)=2^{k-1}+2$, \textcolor{black}{then $\phi_j(x,y,z)$ is not}
absolutely irreducible. For being more specific they proved:

\begin{thm} Suppose $f(x)=x^{2^k+1}+g(x)\in L[x]$ and $\deg(g)=
2^{k-1}+2$. Let $k$ be odd and relatively prime to $n$. If $g(x)$
does not have the form $ax^{2^{k-1}+2}+a^2x^3$ then $\phi$ is
absolutely irreducible, while if $g(x)$ does have this form, then
either $\phi$ is absolutely irreducible or $\phi$ splits into two
absolutely irreducible factors that are both defined over $L$.
\end{thm}

In \textcolor{black}{\cite{DJ, {DJDESI}},}    \textcolor{black}{we} extended these results for the Gold degree
case, \textcolor{black}{we} found new families of polynomials which are not
exceptional APN.

\begin{thm} \label{thm:3mod4}
\textcolor{black}{For} $k \geq 2$, let $f(x)=x^{2^k+1}+h(x) \in L[x]$\textcolor{black}{,}
where $\deg(h)<2^k+1$, and  \textcolor{black}{$\deg(h)\equiv 3\pmod 4$}. Then\textcolor{black}{,} $\phi(x,y,z)$ is absolutely irreducible.
\end{thm}


For the case $1 \pmod 4$, in   \textcolor{black}{ \cite{DJ,{DJDESI}}}  we also proved:

\begin{thm} \label{thm:1mod4}
\textcolor{black}{For $k \geq 2$, let $f(x)=x^{2^k+1}+h(x) \in L[x]$
where $d=\deg(h)\textcolor{black}{\equiv 1} \pmod 4$ and $d<{2^k+1}$. If $\phi_{2^k+1}, \phi_d$ are
relatively prime,  \textcolor{black}{then} $\phi(x,y,z)$ is absolutely irreducible.}
\end{thm}




\textcolor{black}
{In Theorem 10 ( section 4 ), one of our main results is that we prove the relative primeness of the conjecture of APN polynomials $\phi_n(x,y,z)$ and $\phi_m(x,y,z)$  when one of them is  of Gold degree. Thus proving Theorems  5 to  8 and some  other  results in the Gold degree polynomials $\phi(x,y,z)$  unconditionally (see sections 5 and 6).  The case when $d=\deg(h)\textcolor{black}{\equiv 5} \pmod 8$ is  much simpler and   has appeared in   \cite{DJDESI}.
}

The case Kasami-Welch degree polynomials seems to be the hardest one. Rodier proved the following
theorem \cite{Rod2}.

\begin{thm} Suppose that $f(x)=x^{2^{2k}-2^k+1}+g(x) \in L[x]$
where $\deg(g) \leq 2^{2k-1}-2^{k-1}+1$. Let
$g(x)=\sum_{j=0}^{2^{2k-1}-2^{k-1}+1}a_jx^j$. Suppose moreover that
there exist a nonzero coefficient $a_j$ of $g$ such that
$\phi_j(x,y,z)$ is absolutely irreducible. Then $\phi(x,y,z)$ is
absolutely irreducible.
\end{thm}

\vskip .1truein

Rodier also \textcolor{black}{studied} the case when $\deg(g)=2^{2k-1}-2^{k-1}+2$.

\vskip .1truein

\textcolor{black}{We also discuss the relatively prime case of  the Kasami-Welch APN polynomials  with other APN polynomials  in \cite{DJDESI}, and in \cite{IWSDA2015} .}

\section{\textbf{MAIN RESULTS.}}
From now on, let $L=\mathbb{F}_{2^n}$, $\phi(x,y,z)$, $\phi_j(x,y,z)$ as in (2) and (3).
\subsection{\textbf{\textcolor{black}{A Proof of Relatively Prime APN Polynomial  Conjecture} }}

\textcolor{black}{We first give a  proof of  the relatively primeness of the multivariate APN polynomial conjecture, in the Gold degree case, as stated in \cite{DJDESI} and presented in
\cite{IWSDA2015}. For its statement, see Theorem 10 below.}

We begin with the following fact,
due to Janwa and Wilson \cite{Jan}, about the Gold functions.\\
For a Gold number $j=2^k+1$:
\begin{equation}
\phi_j(x,y,z)=\prod_{\alpha \in F_{2^k}-F_2}(x+\alpha y+(\alpha+1)z)
\end{equation}
Let us use the affine transformation $x\leftarrow x+1, y\leftarrow y+1$ on (5). Let's denote
$\widetilde{\phi}_j(x,y)=\phi_j(x+1,y+1,1)$. Then we have

\begin{equation}
\widetilde{\phi}_j(x,y)=\prod_{\alpha \in
\mathbb{F}_{2^k}-\mathbb{F}_2}(x+\alpha y)
\end{equation}

\begin{thm}
\textcolor{black}{If $d$ is an odd integer,  then $\phi_{2^k+1}$ and $ \phi_d$ are
relatively prime for all $k\geq 1$ except when $d=2^l+1$ and $(l,k)>1$.}
\end{thm}


\begin{proof}
Since $(\phi_n,\phi_m)=1 \Leftrightarrow (\widetilde{\phi}_n,\widetilde{\phi}_m)=1$, we will work with the functions $\widetilde{\phi}$.\\
Let $n=2^k+1$, $m=2^il+1$, where $l >1$ is an odd integer.\\
By (6), we will prove the theorem by showing that no term $(x+ay)$ divides $\widetilde{\phi}_m$, for all $a\in \mathbb{F}_{2^k}$, $a \neq 0$, $a \neq 1$. 
Let us supose, by the way of contradiction, that this happen for some $a$, $a \neq 0$, $a \neq 1$. Then $(x+ay)$ divides $f(x,y)=\widetilde{\phi}_m(x,y)(x)(y)(x+y)$
and $f(ay,y)=0$. Writing $f(x,y)$ as a sum of homogeneous terms:
\begin{equation}
f(x,y)=F_{2^i+1}(x,y)+...+F_{m-1}(x,y)+F_{m}(x,y)
\end{equation}

Then $(x+ay)|f(x,y)$ if and only if $(x+ay)$ divides each homogeneous term $F_r$ in (7), implying $F_r(a y,y)=0$. \\
From the expansion of $f$, we have:
\begin{equation*}
F_{m-1}(x,y)=x^{m-1}+y^{m-1}+(x+y)^{m-1}
\end{equation*}
\begin{equation*}
F_{m}(x,y)=x^{m}+y^{m}+(x+y)^{m}
\end{equation*}
Then
\begin{equation}
(ay)^{2^il}+y^{2^il}+(ay+y)^{2^il}=0
\end{equation}
\begin{equation}
(ay)^{2^il+1}+y^{2^il+1}+(ay+y)^{2^il+1}=0 
\end{equation}
which respectively implies that
\begin{equation}
(a+1)^l+a^l+1=0
\end{equation}
\begin{equation}
(a+1)^{l+1}+a^{l+1}+1=0
\end{equation}
Substituting (10) in (11)
$$(a^l+1)(a+1)=a^{l+1}+1$$
$$a^{l+1}+a^l+a+1=a^{l+1}+1$$
\begin{equation}
a^{l-1}=1
\end{equation}
ie, $a$ is a $(l-1)$-th root of unity. Furthermore, using this in (10)
\begin{equation}
(a+1)^{l-1}=1
\end{equation}
ie, $a+1$ is also a $(l-1)$-th root of unity.\\
Now, let us consider the term $F_{m-(2^i+1)}$ in (7) and prove that $x+ay$ does not divide it.\\
$F_{m-(2^i+1)}(x,y)$=$m \choose {2^i+1}$$(x^{m-(2^i+1)}+y^{m-(2^i+1)}+(x+y)^{m-(2^i+1)})$

The classical theorem of Lucas states:

For non-negative integers $a$, $b$ and a prime $p$, the following relation holds:\\
$a\choose b$ $\equiv \prod _{i=0}^{r}$$a_i \choose b_i$$\mod p$,\\
where\\
$a=a_rp^r+a_{r-1}p^{r-1}+...+a_1p+a_0$,\\
$b=b_rp^r+b_{r-1}p^{r-1}+...+b_1p+b_0$,\\
are the base $p$ expansions of $a$ and $b$ respectively. (Where by   convention $a\choose b$$=0$ if $a<b$ and ${0 \choose 0}=1$.)

Let $l=a_r2^r+a_{r-1}2^{r-1}+...+a_12+1$ be the base 2 expansion of $l$. Then, the expansion of
$m$ is $m=2^il+1=a_r2^{i+r}+a_{r-1}2^{i+r-1}+...+a_12^{i+1}+2^i+1$.
Using the  theorem of Lucas, we have that $m \choose {2^i+1}$=1.
\\ \\
For $x+ay$ to divide $F_{m-(2^i+1)}$, it should happen that $F_{m-(2^i+1)}(ay,y)=0$, however:
\begin{center}
$F_{m-(2^i+1)}(ay,y)=(ay)^{m-(2^i+1)}+y^{m-(2^i+1)}+(ay+y)^{m-(2^i+1)}$\\
$=((a+1)^{m-(2^i+1)}+a^{m-(2^i+1)}+1)y^{m-(2^i+1)}$\\
$=((a+1)^{2^i(l-1)}+a^{2^i(l-1)}+1)y^{m-(2^i+1)}$\\
\end{center}
using that $a,a+1$ are $(l-1)$-th roots of unity, we get that the last equality is equal to $y^{m-(2^i+1)}$. A contradiction.

\end{proof}

\section{Application to  Absolutely Irreducible Polynomials and Exceptional APN Functions}
\label{sec:application:conjecture}

\noindent Using this resulst we can generalize theorems 7 and 8 in the following theorem.

\begin{thm}
\textcolor{black}
{For $k \geq 2$, let $f(x)=x^{2^k+1}+h(x) \in L[x]$
where $\deg{h}<2^k+1$, and  $\deg(h)$ is an odd integer (not a Gold number $2^l+1$ with $(l,k)>1$). Then $\phi(x,y,z)$ is absolutely irreducible,  and $f(x)$ can not be exceptional APN.
}\end{thm}


\subsection{\textcolor{black}{Some Pending Cases.}}
\textcolor{black}{From  theorem 11, the  missing cases are where  Gold degree polynomials of the form $f(x)=x^{2^k+1}+h(x)$,  with $\deg(h)$   any  gold number}. However, for polynomials
of the form $f(x)=x^{2^k+1}+h(x)$, $\deg(h)=2^{k'}+1$, $(k,k')=1$; $\phi_{2^k+1},\phi_{2^{k'}+1}$ are relatively primes.
Then $\phi(x,y,z)$ is absolutely irreducible by theorem 8.

For non relatively prime numbers $k,k'$, as in the proof of the first case of theorem 8 (see \cite{DJ}), we have that: $Q_{t-1}=0,Q_{t-2}=0,..., Q_{1}=0,Q_0=0$ (Observe in the proof that $t < e$, where $e=2^k+1-d$).

Then, the hypersurface $\phi(x,y,z)$ related to $f$ satisfies:
\begin{equation*}
\sum_{j=3}^{2^{k}+1}a_j\phi_j(x,y,z)=(P_s+P_{s-1}+...+P_0)(Q_t)
\end{equation*}

Therefore, $\phi(x,y,z)$ would be absolutely irreducible if $h$ contains any term of degree $m$ such that
$\phi_{2^k+1},\phi_m$ are relatively primes. This condition is best posible in the sense that if $h$ does not have such a term, then 
$\phi(x,y,z)$ is not more absolutely irreducible. Theorem 10 of the previous section and the comments at the begining of this subsection provides many examples for this condition to happen, almost for any odd number.

Until now, almost all the found families of Gold degree polynomials that fails to be exceptional APN are of the form $x^{2^k+1}+h(x)$ for an odd degree
of $h$. Theorem 5 of section 3 justifies this fact. In the next subsection we will discuss the case when $\deg (h)$ is an even number.

\section{\textcolor{black}{On the boundary of theorem 5.}}


Next,  we prove the version of theorem 6 for the even case.

\begin{thm} For $k \geq 2$, let $f(x)=x^{2^{2k}+1}+h(x)\in L[x]$
where $\deg (h)=2^{2k-1}+2$. Let
$h(x)=\sum_{j=0}^{2^{2k-1}+2}a_jx^j$. If there is
a nonzero coefficient $a_j$ such that $(\phi_{2^{2k}+1},\phi_j)=1$.
Then $\phi$ is absolutely irreducible.
\end{thm}

\begin{proof}
Suppose by contradiction that $\phi$ is not absolutely irreducible. then
 $\phi(x,y,z)=P(x,y,z)Q(x,y,z)$, where
$P,Q$ are non-constants polynomials defined on some extension of $L$.
 Writing $P,Q$ as a sum of homogeneous
terms:
\begin{equation}
\sum_{j=0}^{2^{2k}+1}
a_j\phi_j(x,y,z)=(P_s+P_{s-1}+...+P_0)(Q_t+Q_{t-1}+...+Q_0)
\end{equation}
where $P_i,Q_i$ are zero or homogeneous of degree $i$,
$s+t=2^{2k}-2$. Assuming \textcolor{black}{without loss of generality} that $s\geq t$, then
$2^{2k}-2> s\geq \frac{2^{2k}-2}{2} \geq t>0$.\\
From the equation (14) we have that:
\begin{equation}
P_sQ_t=\phi_{2^{2k}+1}.
\end{equation}
Since by (5), $\phi_{2^{2k}+1}$ is equal to the product of different
linear factors, then $P_s$ and $Q_t$ are relatively primes. \textcolor{black}{Since $h(x)$
is assumed to have degree $2^{2k-1}+2$}, the homogeneous terms of degree $r$, for
$2^{2k-1}-1< r <2^{2k}-2$, are equal to zero. Then equating the
terms of degree $s+t-1$ gives $P_sQ_{t-1}+P_{s-1}Q_t=0$. Hence we
have that $P_s$ divides $P_{s-1}Q_t$ and this implies that $P_s$
divides $P_{s-1}$, since $P_s$ and $Q_t$ are relatively primes. We
conclude that $P_{s-1}=0$ since the degree of $P_{s-1}$ is less than
the degree of $P_{s}$. Then we also have that $Q_{t-1}=0$ as $P_s\neq 0$.\\
Similarly, equating the terms of degree $s+t-2, s+t-3,...,s+1$ we
get:
$$P_{s-2}=Q_{t-2}=0, P_{s-3}=Q_{t-3}=0,...,P_{s-(t-2)}=Q_1=0$$
The \textcolor{black}{(simplified) equation} of degree $s$ is:
\begin{equation}
P_sQ_0+P_{s-t}Q_t=a_{s+3}\phi_{s+3}
\end{equation}
\textcolor{black}{Lets} consider two cases to prove the absolute irreducibility of
$\phi(x,y,z)$.\\
\textbf{First case}: $s > t$.\\
Then $s > 2^{2k-1}-1$ and $\phi_{s+3}=0$. Then the equation (16)
becomes:
$$P_sQ_0+P_{s-t}Q_t=0.$$
Then, using the previous argument, $P_{s-t}=Q_0=0$. It means that
$Q=Q_t$ is homogeneous. By the equations (17), (15) we have that for
all $j$, $\phi_j(x,y,z)$ is divisible by $x+\alpha y+(\alpha+1)z$
for some $\alpha \in \mathbb{F}_{2^k}-\mathbb{F}_2$, which is a
contradiction by the
hypothesis of the theorem.\\
\textbf{Second case}: $s=t=2^{2k-1}-1$\\
For this case the equation (16) becomes:
\begin{equation}
P_sQ_0+P_0Q_t=a_{s+3}\phi_{2^{2k-1}+2}.
\end{equation}
If $P_0=0$ or $Q_0=0$, then we have that $Q=Q_t$ or $P=P_s$. Then by
similar arguments of the first case we have a contradiction. If both
$P_0, Q_0$ are different from zero, let us consider the intersection
of $\phi$ with the line $z=0,y=1$. Then the equation (15) and (17)
become:
\begin{equation}
P_sQ_t=\prod_{\alpha \in \mathbb{F}_{2^{2k}}-\mathbb{F}_2}(x+\alpha)
\end{equation}

\begin{equation}
P_sQ_0+P_0Q_t=a_{s+3}(x+1)(x)\prod_{\alpha \in
\mathbb{F}_{2^{2k-2}}-\mathbb{F}_2}(x+\alpha)^2
\end{equation}
This last equation comes from the fact that
$$\phi_{2^{2k-1}+2}=(x+y)(x+z)(y+z)\phi_{2^{2k-2}+1}^2$$
It is easy to show that
$\mathbb{F}_{2^{2k}}\cap\mathbb{F}_{2^{2k-2}}=\mathbb{F}_{2^2}$. Let
$x=\alpha_0\in \mathbb{F}_{4}$, $\alpha_0 \neq 0,1$. Then from (18)
we have that $P_s(\alpha_0)=0$ or $Q_t(\alpha_0)=0$.\\
If $P_s(\alpha_0)=0$, then $Q_t(\alpha_0) \neq 0$ (since $P_sQ_t$ is a product of
different linear factors) and from equation (19) we have $P_0Q_t(\alpha_0)=0$
that is a contradiction sice both $P_0,Q_t(\alpha_0)$ are different from zero.
The case $Q_t(\alpha_0)= 0$ is analogous. Therefore $\phi(x,y,z)$ is
absolutely irreducible.
\end{proof}
\noindent One of the families covered by this theorem is:\\
$f(x)=x^{17}+h(x)$ where $\deg(h)=10$, except the case
$f(x)=x^{17}+a_{10}x^{10}+a_5x^5$, $a_{10}\neq 0, a_5 \neq 0$.

\textcolor{black}{{\bf COMMENT:} Theorem 10 can be applied to Theorem 12 to prove absolute irreducibility of many families of this kind.}
\textcolor{black}{We also discuss the relatively prime case  the Kasami-Welch with other APN polynomials  in \cite{DJDESI},   and  in  \cite{IWSDA2015}.}



\begin{thebibliography}{1}
\bibitem{Aub} Y. Aubry, G. McGuire, F. Rodier. \emph{A Few More Functions That Are
Not APN Infinitely Often}.\hskip 1em plus 0.5em minus 0.4em\relax
2009.

\bibitem{Book:Berlekamp}
E.R. Berlekamp, {\em Algebraic Coding Thoeory.} McGraw Hill, 1968,

\bibitem{Ber} T.P. Berguer, A. Canteaut, P.Charpin, Y. Laigle-Chapuy, \emph{On almost
Perfect Nonlinear Functions Over $F_{2n}$}. \hskip 1em plus
  0.5em minus 0.4em\relax 2006.
\bibitem{Bud} L. Budaghyan, C. Carlet, G. Leander. \emph{Constructing New APN
functions from known ones}, \hskip 1em plus
  0.5em minus 0.4em\relax 2007.
\bibitem{Byr} E. Byrne, G. McGuire, \emph{Quadratic Binomial APN Functions and
Absolutely Irreducible Polynomials},\hskip 1em plus
  0.5em minus 0.4em\relax 2008.
\bibitem{Car} C. Carlet, P. Charpin, V. Zinoviev, \emph{Codes, bent functions and
permutations suitable for DES-like cryptosystems, Designs, Codes and
Cryptography}, \hskip 1em plus
  0.5em minus 0.4em\relax 15, (1998) 125-156.
\bibitem{CauF} F. Caullery \emph{Polynomial functions of degree 20 which are APN infinitely often},
\hskip 1em plus 0.5em minus 0.4em\relax arXiv:1212.4638v2[cs.IT] (25 Jan 2013).
\bibitem{DJ} M. Delgado, H. Janwa, \emph{\it On The Conjecture on APN
Functions},\hskip 1em plus
  0.5em minus 0.4em\relax arXiv:1207.5528v1[cs.IT] (Jul 2012).
 \bibitem{IWSDA2015} 
 {M. Delgado and H. Janwa,``Progress on the Conjecture on APN functions in Absolutely Irreducible Polynomials,'' IWSDA2015, December 2015. 
 http://www.slideshare.net/MoisesDelgadoOlorteg/iwsda2015talk16sept2015
}

\bibitem{DJDESI}
\textcolor{black}{M. Delgado, H. Janwa, \emph{\it On the  Conjecture on  APN  Functions  and\\  \textcolor{black}{ Absolute Irreducibility of Polynomials,}
Designs, Codes and Cryptography}, \hskip 1em plus
  0.5em minus 0.4em\relax  Springer-Verlag First http://ww.doi.org/DOI: 10.1007/s10623-015-0168-1  (February 2016). 
   Received 2 April 2015;
Revised 2 December 2015;
Accepted 9 December 2015. Printed version scheduled to appear.}

\bibitem{Ede} Y. Edel, G. Kyureghyan, A. Pott, \emph{A New APN Function that Is
not Equivalent to a Power Mapping}, \hskip 1em plus
  0.5em minus 0.4em\relax 2006.
\bibitem{Ful} W. Fulton, \emph{Algebraic Curves}, \hskip
1em plus
  0.5em minus 0.4em\relax Benjamin, New York, 1969.
\bibitem{Her} F.Hernando, G. McGuire \emph{Proof of a Conjecture on Sequence of
Excepcional Numbers},\hskip 1em plus
  0.5em minus 0.4em\relax Classifying Cyclic Codes and APN Functions.
2009.
\bibitem{Jan} H. Janwa, M. Wilson, \emph{Hyperplane Sections of Fermat
Varieties in $P^{3}$ in Char. $2$ and Some Applications to Cyclic
Codes}. Lecture Notes in Computer Science. \hskip 1em plus
  0.5em minus 0.4em\relax Springer Verlag. 1993.
\bibitem{Jan2} H. Janwa, G. McGuire, M. Wilson \emph{Double Error-correcting Cyclic
Codes and Absolutely Irreducible Polynomials over $GF(2)$},\hskip
1em plus
  0.5em minus 0.4em\relax
 Journal of Algebra 1995.
\bibitem{Rod} F. Rodier. \emph{Borne sur le Degre des Polynomes presque
Parfaitement Non-lineaires},\hskip 1em plus
  0.5em minus 0.4em\relax 2008.
\bibitem{Rod2}  F. Rodier \emph{Some more functions that are not APN
infinitely often. The case of Kasami exponents}, \hskip 1em plus
  0.5em minus 0.4em\relax Hal-00559576, version 1-25 (Jan 2011).

\bibitem{AGCT2013}\textcolor{black}{ M. Delgado and H. Janwa, ``Further Results on Exceptional APN Functions,'' 
http://www.math.iitb.ac.in/\~srg/AGCT-India-2013/Slides/HeeralalJanwa.pdf.}




\end{thebibliography}
\end{document}